\documentclass[review, authoryear]{elsarticle}
\usepackage[utf8]{inputenc}

\usepackage[table]{xcolor}
\definecolor{gray1}{gray}{0.95}
\definecolor{gray2}{gray}{0.975}
\usepackage{booktabs}
\usepackage{multirow}

\usepackage{mathtools}
\usepackage{amssymb}
\usepackage{amsthm}

\newtheorem{assumption}{Assumption}

\usepackage[colorlinks=true, citecolor=red, breaklinks=true]{hyperref}

\journal{arXiv}
\newcommand{\ts}{\textsuperscript}
\newcommand{\st}{\text{subject to}}
\renewcommand{\leq}{\leqslant}
\renewcommand{\geq}{\geqslant}
\newcommand{\eps}{\varepsilon}

\renewcommand{\phi}{\varphi}
\newcommand{\outarcs}{\delta^{\text{out}}}
\newcommand{\inarcs}{\delta^{\text{in}}}

\usepackage{algorithm}
\usepackage[noend]{algpseudocode}

\newtheorem{lemma}{Lemma}[section]

\theoremstyle{definition}

\usepackage[margin=1in]{geometry}

\usepackage{graphicx}
\graphicspath{ {./images/} }

\usepackage{subcaption}
\DeclareCaptionFormat{custom}
{\textbf{#1#2}\textit{\small #3}}

\begin{document}

% Front matter
\begin{frontmatter}
\title{Logic-Based Benders Decomposition for Wildfire Suppression}

% Authors
\author{Mitchell G Harris$^\dagger$}
\author{\quad Michael A Forbes}
\author{\quad Thomas Taimre}
\address{School of Mathematics and Physics, The University of Queensland, Australia}

% Abstract
\begin{abstract}
We study the problem of locating fire suppression resources in a burning landscape in order to minimise the total area burned.
The landscape is modelled as a directed graph, with nodes representing regions of the landscape, and arcs representing adjacency relationships.
The fire spread is modelled using the minimum travel time principle.
We propose a non-linear integer programming formulation and an exact solution approach utilising logic-based Benders decomposition. We benchmark the approach against a mixed integer program and an iterated local search metaheuristic from the literature. We are able to solve challenging instances to proven optimality in a reasonable amount of time.

\bigskip
\end{abstract}

\begin{keyword}
	integer programming; wildfires; network interdiction; logic-based Benders decomposition
\end{keyword}

\end{frontmatter}
\noindent Accompanying GitHub: \href{https://tinyurl.com/fs6yjwsx}{https://tinyurl.com/fs6yjwsx}

% Corresponding author
%\noindent $^\dagger$Corresponding Author email: \href{m.g.harris@uq.edu.au}{m.g.harris@uq.edu.au}
\noindent $^\dagger$Corresponding Author email: \href{mitchopt@gmail.com}{mitchopt@gmail.com}

% Formatting
\allowdisplaybreaks

% Mainmatter
\bigskip
\section{Introduction}
\label{sec:Intro}Wildfires can have devastating consequences for wildlife, the environment, and human lives.
The 2019/20 Australian wildfires covered an area greater than 17 million hectares, destroyed 3094 houses, and led to at least 43 direct deaths; \cite{Parliament}.
Conservative estimates moreover suggest that at least one billion mammals, birds, and reptiles were killed.
Wildfires also lead to significant greenhouse gas emissions, contributing to climate change.
For these and other reasons, governments in Australia and around the world invest heavily in various suppression resources and fire fighting strategies, with the goal of mitigating the damage caused by wildfires.

In this paper we study the problem of locating suppression resources in a burning landscape to minimise the area burned.
Before describing the problem, we clarify some notation;
$\mathbb{R}_+$ denotes the non-negative reals, while $\mathbb{R}_{++}$ denotes the positive reals; likewise $\mathbb N_+$ and $\mathbb N_{++}$ for integers.
We use $(N, A,\, c)$ as shorthand for a directed graph with sets $N$ of nodes, $A$ of arcs, and weights $c_a \in \mathbb{R}$ for $a \in A$. We use ``network'' as a synonym for ``directed graph.''
A superscript $*$ denotes an optimal solution to an optimisation problem, and $^\prime$ a specific feasible, but not necessarily optimal solution.
The wedge operator $\wedge$ represents logical ``and.''

\subsection{Problem Description}\label{subsec:problemdescr}

We are given a sparse network $(N, A)$, where
the nodes $N$ represent regions in a landscape, and the arcs $A \subset N\times N$ are adjacency relationships.
In practice $(N, A)$ arises from a partition of the underlying landscape into a union of polygonal regions.
In principle, we can accommodate a cellular decomposition of any topology and resolution.
For each arc $ij \in A$, a weight $c_{ij} \in \mathbb{R}_{++}$ is the time taken for the fire to spread from node $i$ to the adjacent node $j$.
The time taken for fire to spread from a node $i \in N$ to any other node $j \in N$ is equal to the weight of the shortest path from $i$ to $j$.
We are given a subset $I \subset N$ of ignition nodes, and say that the fire arrived at the ignition nodes at time $0$. We are given a constant $\psi \in \mathbb{R}_{++}$ and we say that a node is protected if the fire arrives at that node strictly later than $\psi$. 
We are given a finite set $T \subset [0, \psi)$ of times, and for each $t \in T$, up to $a_t \in \mathbb{N}_{++}$ suppression resources become available at that time.
Another constant, $\Delta \in \mathbb{R}_{++}$, represents the delay induced by a resource.
In other words, if there is a resource at node $i \in N$, then the length of every arc that leaves $i$ is increased by $\Delta$.
We make the following additional assumption:
\begin{assumption}\label{assume:notonfire}
We may not locate a resource at a node that is already on fire.
\end{assumption}
\noindent
The threshold $\psi$, which we call the arrival time target, may represent the time taken for people to evacuate, for additional fire crews to arrive, or something else.
In any case, the goal is to maximise the number of protected nodes; or equivalently to minimise the number of unprotected nodes.

We can think of the problem as a two-person attacker--defender game.
The defender (\textit{us}) makes a set of interdiction decisions; that is, we locate suppression resources.
The attacker (\textit{the fire}) then solves a shortest path tree problem in the resulting network.
The goal of the defender is to make interdiction decisions that minimise the number of nodes reached by an optimal attack within the threshold time.
An interdiction policy can be considered feasible only if it satisfies Assumption \ref{assume:notonfire} with respect to the optimal attack.

Consider the following example with $16$ nodes in a square grid. For each node there is an arc to each of its orthogonally adjacent nodes. There is a single ignition node. We have $\psi = 28$ and $\Delta = 50$. The node labels are the fire arrival times. The left tree shows arrival times with no interdiction. On the right, two suppression resources become available after $5$ minutes.  Allocating them (\textit{red nodes}) changes the arrival times and and fire paths to some nodes, ultimately protecting four of them.

\medskip
\begin{center}
	\includegraphics[width=.25\linewidth]{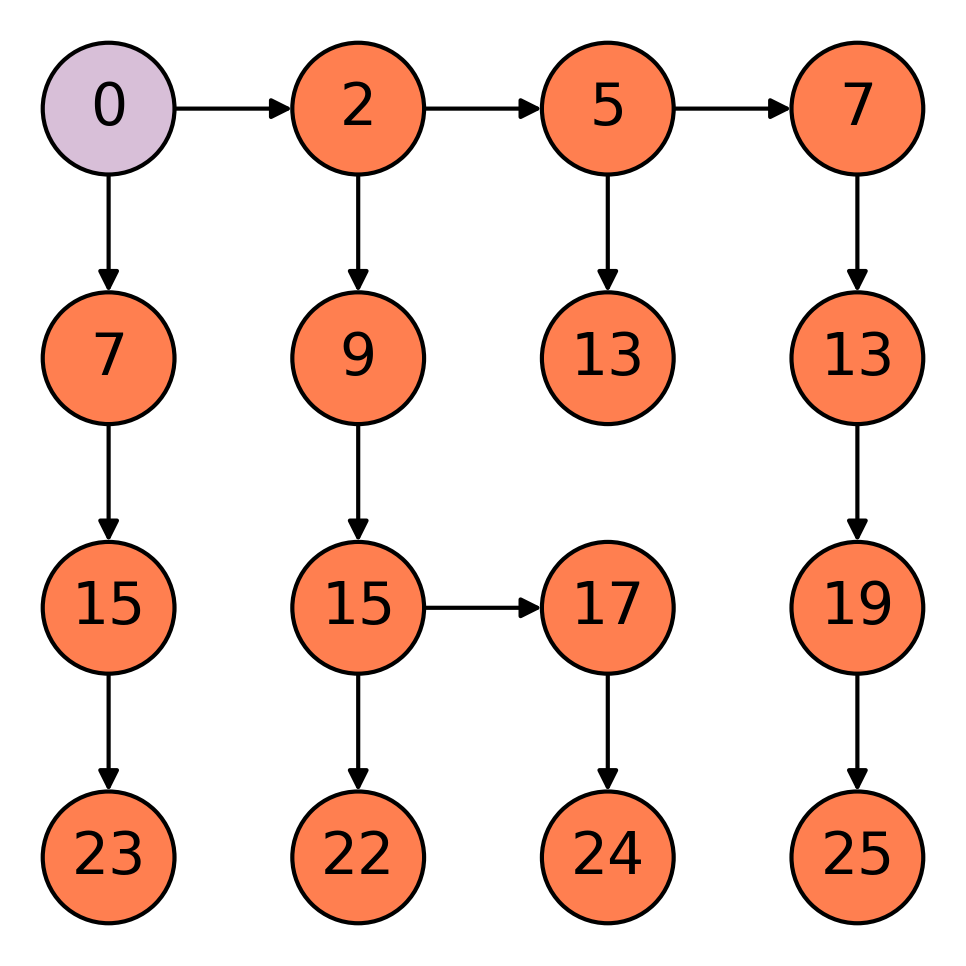}\qquad\qquad
	\includegraphics[width=.25\linewidth]{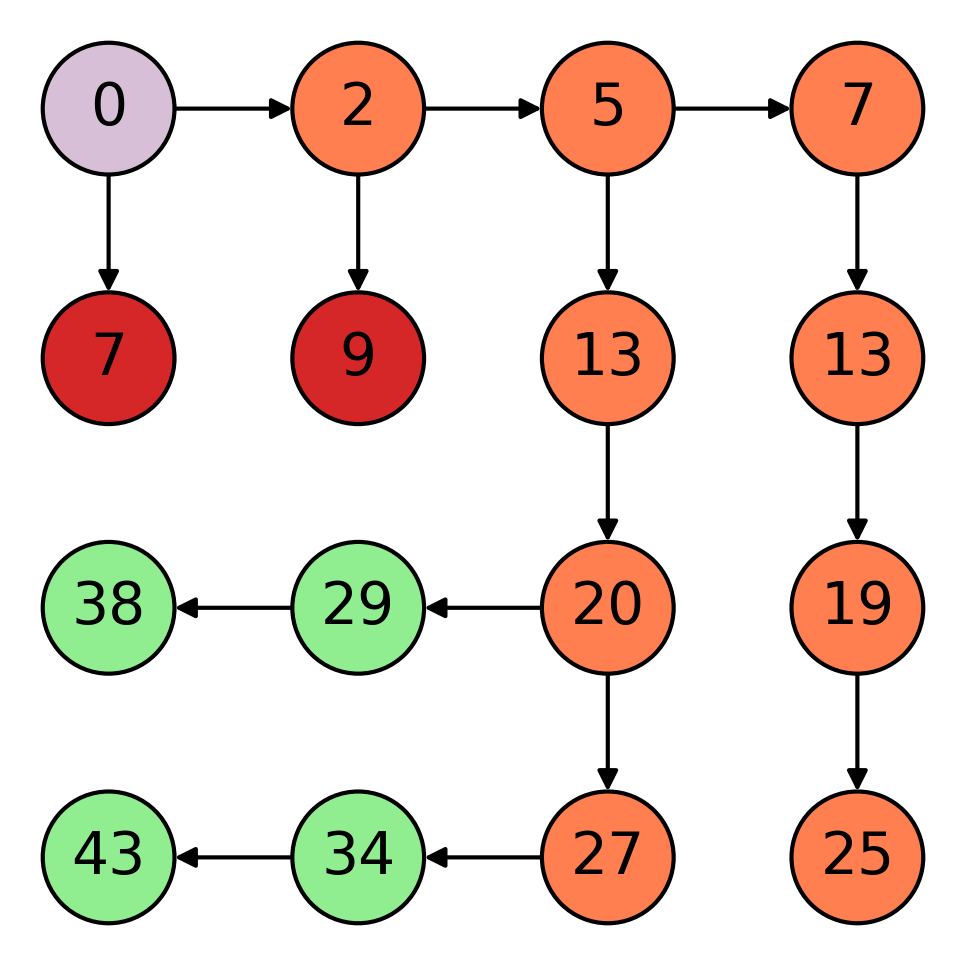}
\end{center}

A clever mixed integer programming (MIP) model of this problem was formulated by \cite{Alvelos2018}.
It is well-known that the optimal dual variables for the shortest path tree problem correspond to the distances of nodes from the root node; \cite{KarupRorbech2004}.\footnote{To be more precise; \textit{if} the dual problem is feasible, then \textit{there exists} a dual-optimal solution such that this is the case. We can retrieve it by setting the dual variable associated with the root node to zero; i.e. constraint (\ref{mip:con4}) in the appendix.}
Thus Alvelos' basic idea was to embed both the shortest path tree problem and its dual into the model.
Explicit slack variables and additional constraints enforce complementary slackness, so that we have access to the correct fire arrival times regardless of the objective function.
That study considered small landscapes with $36$ nodes,
but it is not hard to see that this MIP is intractable for instances of moderate size.
More recently \cite{MendesAlvelos2022} proposed an iterated local search (ILS) metaheuristic that provides good feasible solutions.

In this paper, we propose a novel, exact non-linear integer programming (NLIP) formulation, and a logic-based Benders decomposition (LBBD) solution approach.
We generate synthetic instances using the same parameters as Alvelos and Mendes, and find that LBBD outperforms the direct MIP by orders of magnitude.
We also replicated the ILS metaheuristic in order to compare it to LBBD on larger instances.
Our results indicate that the LBBD approach outperforms both the MIP and the metaheuristic.

In the next section, we will review some of the relevant literature.

\subsection{Literature Review}

\subsubsection{Models of Wildfire Suppression}

Mathematicians have been interested in analytical models of fire spread since at least the 1940s; see \cite{CurryFons1940},  \cite{Fons1946} for instance.
Operations research practitioners in particular have been interested in the optimal fighting of fires since at least the $1960$s; \cite{Jewell1963}. \cite{MendesAlvelos2022} conduct a systematic review of optimisation models related to fire suppression. We mention a representative sample here, and refer the reader to that study for a more thorough treatment. 

\cite{MarianovReVelle1992} propose set covering models that locate resources in order to maximise the population covered by a so-called ``standard response.''
\cite{MeesStrauss1992} propose a probabilistic allocation model which allocates resources to demand points in order to maximise expected utility.
\cite{Wiitala1999} propose a dynamic programming algorithm that minimises the cost of fire containment, where the fire is represented by a fixed perimeter parameterised by the current time.
In that study, among others, the fire is said to be contained when the line built by the resources (the \textit{fireline}) encircles the perimeter of the fire.
\cite{RachaniotisPappis2006} consider a scheduling approach where fires are jobs that need to be processed by available resources; fire and resource behaviour is encoded in a scheduling-theoretic manner.
\cite{HomChaudhuri2010} propose a genetic algorithm for determining the parameters of a parametric equation representing the fireline.
\cite{MerweMinasOzlenHearne2015} propose a vehicle routing formulation that sends vehicles from depots to maximise the total value of protected assets. 

\cite{Hartnell1995} introduced the so-called firefighter problem. An outbreak of fire occurs at some vertices in a graph, and in each time step the fire spreads from each burned vertex to all of its undefended neighbours. In each time step firefighters can defend a fixed number of vertices. The fire is said to be contained when there is no burned vertex with a neighbour that is undefended or unburned. The goal is to compute a defensive sequence which minimises the number of burned vertices. This problem has enjoyed significant theoretical and practical attention. According to \cite{WangMoeller2002}, Hartnell, Finbow, and Schmeisser proved that two firefighters per time step is sufficient to contain an outbreak from a single vertex in the $2$-dimensional square grid. Wang and Moeller showed that a single outbreak in a regular graph of degree $r$ can be contained with $r-1$ firefighters per step, and that $2d - 1$ firefighters per step are sufficient in the $d$-dimensional square grid for $d$ at least three. \cite{DevelinHartke2007} proved---among other interesting things---that the bound is tight. More recently \cite{RamosSouzaRezende2020} proposed a novel metaheuristic that outperformed existing solution techniques for a set of benchmark instances.

\cite{Finney2002} proposed the minimum travel time principle; that is, the time taken for a node to be burned is equal to the length of the shortest path from an ignition node to that node. They convincingly argued that if the travel times of arcs are well-estimated, then then the fire dynamics obtained are almost identical to those of more sophisticated differential equations-based methods.
The accurate estimation of the parameters is, however, beyond the scope of our paper.
\cite{WeiRideoutHall2011} proposed a MIP model based on the minimum travel time principle that aims to minimise the value of cells that burn within a set time.
\cite{BelvalWeiBevers2015, BelvalWeiBevers2016} describe (\textit{respectively}) deterministic and stochastic MIP models that incorporate both detrimental and ecologically beneficial fire into the objective function. The latter study allows for variance in the wind over time. \cite{BelvalWeiBevers2019} extend the model with additional realism, such as crew safety and logistics.

\subsubsection{Benders Decomposition}

Classical Benders decomposition was introduced by \cite{Benders1962} and we refer the reader to \cite{Rahmaniani2017} for a survey of applications. In classical Benders decomposition, a MIP is partitioned into an integer master problem and a linear subproblem. New master variables estimate the contribution of the subproblem to the objective. At incumbent solutions to the master problem, the subproblem is solved to generate Benders cuts, which refine the master estimates. Benders cuts are derived from the linear programming dual of the subproblem. \cite{Geoffrion1972} introduced generalized Benders decomposition, which is an extension of the technique to certain non-linear programming problems that uses convex duality theory instead. Another variant, combinatorial Benders decomposition, was introduced by \cite{CodatoFischetti2006}, and is suited to problems with implication constraints.

\cite{Hooker2000} and \cite{HookerOttosson2003} introduced LBBD, which is an ambitious generalisation of classical Benders decomposition. Like Geoffrion, Codato, and Fischetti, they observed that the concept underlying Benders decomposition could be applied in much more general settings. Since then, LBBD has enjoyed remarkable success on a diverse range of problems. Like its classical counterpart, the emerging LBBD literature is vast and broad. Thus we refer the reader to \cite{Hooker2019} for a survey of applications.

Unlike classical Benders decomposition, with LBBD, the subproblem can take the form of any optimisation problem.
In fact, the subproblem can be any function evaluation, as long as suitable Benders cuts can be derived.
While this flexibility has allowed the principles of Benders decomposition to be applied to a larger set of problems, one weakness is that the structure of the Benders cuts is problem-specific. Indeed, finding suitable Benders cuts sometimes requires considerable effort, since we lack a convenient theory of duality for an arbitrary subproblem.

\subsubsection{Network Interdiction}

Interdiction problems are a class of two-player attacker--defender games. An interdiction refers to an action that inhibits the attacker(or defender)s operation. One player---the \textit{leader}---gets to make a set of interdiction decisions (subject to constraints). The \textit{follower} then solves an optimisation problem parameterised by the leader decisions. The goal of the leader is to minimise the maximum (or maximise the minimum) objective value that can be obtained by the follower. Interdiction problems are thus a special case of the more general class of Stackelberg games; see \cite{Stackelberg1952}. Classical examples include the \textit{shortest path, maximum flow, minimum-cost flow, clique, matching}, and \textit{knapsack} interdiction problems. We refer the reader to \cite{SmithSong2020} for an introduction and an extensive literature review.

Most relevant to us are shortest path problems.
\cite{FulkersonHarding1977} considered the linear shortest path interdiction problem, showing that it is equivalent to the classical minimum-cost flow problem. \cite{IsraeliWood2002} considered the version with binary interdiction variables, which are more realistic in practice, but render the problem NP-Hard. They propose a classical Benders decomposition algorithm, which they augment with so-called ``super-valid inequalities'' that are analogous to logic-based Benders feasibility cuts. In fact, they make a similar logical observation to one we exploit in this paper; that if the interdictor is to force the shortest path to be longer than the incumbent, they must interdict at least one more edge on that path.

Interdiction against classical graph optimisation problems is a well-studied area. But interest in interdiction against more complex subproblems is growing. We believe there is significant scope for the application of LBBD to interdiction against follower problems with integer variables. For example, \cite{Enayaty2019} applied LBBD to a multi-period problem, where interdictions need to be scheduled in order to minimise the cumulative value of maximum flows over a finite time horizon. 
There the MIP master problem makes interdiction decisions, which are scheduled by the constraint programming subproblem.

\subsection{Outline of the Paper}

In Section \ref{sec:LBBDOutline} we outline the LBBD method.
In Section \ref{sec:LBBD} we describe an LBBD formulation of the main problem. In Section \ref{sec:compExperiments} we outline the results of some computational experiments. We conclude in Section \ref{sec:conclusion} with a discussion of the results and some avenues for future research.
In the appendix, we will describe a strengthened form of the MIP from \cite{Alvelos2018} that features in the benchmarking.

\section{Outline of Logic--Based Benders Decomposition}
\label{sec:LBBDOutline}In this section we outline the LBBD method. To be precise, we reiterate and refine the general mathematical framework proposed in \cite{Forbes2021}.
Consider an optimisation problem of the following form, which we call the \textit{original problem} (OP):
\begin{equation}
	\begin{aligned}
		\min_z \quad & g(z) + \sum_{\omega \in \Omega}\Theta_{\omega}(z) \\
		\st \quad & z \in Z,
	\end{aligned}\label{original_problem}
\end{equation}
where $Z, Z'$, and $\Omega$ are sets,
and
\begin{equation*}
g : Z' \to \mathbb{R}\cup\{\pm\infty\} \quad\text{and}\quad \Theta_{\omega} : Z' \to \mathbb R \cup\{\pm\infty\} \quad\text{for }\omega \in \Omega
\end{equation*}
are extended-real valued functions.
Furthermore, $Z \subseteq Z'$, $Z'$ is the domain of the objective function, and $\Omega$ is finite.
We adopt the convention whereby a minimisation problem has $\infty$ objective value if it is infeasible, and $-\infty$ if it is unbounded.
Intuitively speaking, we think of $g$ as the \textit{easy} part of the objective function, while $\Theta_\omega$ are the \textit{hard} parts, such as:
\begin{enumerate}
	\item[(1)] The value function of another minimisation problem (the subproblem) parameterised by the $z$ variables,
	\item[(2)] An indicator function for a logical proposition; in other words, $\Theta_\omega(z) = \infty$ if $\textsf{Prop}(z)$ is true, and $\Theta_\omega(z) = 0$ if $\textsf{Prop}(z)$ is false,
	\item[(3)] The output of a measurement, simulation, or function evaluation.
\end{enumerate}

Our soft assumption is that the $\Theta_\omega$ functions either can't be embedded directly into our underlying optimisation framework, or---if they can---that doing so results in a prohibitively difficult or even intractable model.
Note, moreover, that optimisation problems are not typically expressed in the form of OP a priori.
Indeed, in this framework, the ``decomposition'' aspect of Benders decomposition lies precisely in the expression of the problem in this unique form. 
For simplicity we will refer to $\Theta_{\omega}$ as ``subproblems'' whether or not they correspond to optimisation problems per se.

An important special case occurs if $Z$ is the intersection of a polytope with the integer lattice, and $g$ is a linear function.  In other words, if the problem is
\begin{equation}
	\begin{aligned}
		\min_z \quad & c^\top z + \sum_{\omega \in \Omega}\Theta_{\omega}(z) \\
		\st \quad & Az \leq b, \,z \in \mathbb{Z}^n, \\
		& z_{j\min} \leq z_j \leq z_{j\max} \quad \forall j \in \{1, \ldots, n\},
	\end{aligned}\label{almostIP}
\end{equation}
for some positive integer $n$, matrix $A$, vectors $b$ and $c$, and constants $z_{j\min} \leq z_{j\max}$ for $j \in \{1, \ldots, n\}$.
The problem to be solved now is \textit{almost} an integer program, save for the inclusion of the complicating $\Theta_\omega$ functions.
For example, in classical Benders decomposition, $\Theta_\omega$ is the value function of a linear program (LP) parameterised by the integer variables.
Traditionally the integer variables have been understood to be the complicating factors.
But although LPs are easier to solve than IPs in theory and practice, it is the fact that we carry an LP for every $z \in Z$ that complicates an otherwise tractable integer program.

To apply Benders decomposition to OP we first ask ourselves the following question:
\begin{itemize}
	\item[(Q)] For each $\omega \in \Omega$ and $z' \in Z$, can we find a function $B_{\omega z'} : Z \to \mathbb{R}$  such that
	\begin{equation}
		B_{\omega z'}(z) \leq \Theta_{\omega}(z)\label{eq:benders_qn}
	\end{equation}
	is a valid inequality that obtains equality at $z = z'$, where $B_{\omega z'}$ can be encoded in our underlying optimisation framework?
\end{itemize}
If so, then the problem is a candidate for Benders decomposition. 
Since no reasonable optimisation framework can tolerate singularities in the constraints, two observations need to be made. First, that if $\Theta_\omega(z') = -\infty$ for any $\omega \in \Omega$, then OP is unbounded, and we can stop without looking for $B_{\omega z'}$. Second, that if $\Theta_\omega(z') = \infty$ for any $\omega \in \Omega$, then the subproblem associated with $(\omega, z')$ is infeasible. Therefore $z'$ cannot be an optimal solution to OP, and it is enough to find inequalities that separate $y'$ from the set of feasible solutions.

To apply Benders decomposition in practice, we first replace the $\Theta_\omega$ functions with a new collection of continuous variables, $\theta_\omega$, intended to estimate their contribution to the objective function. The resulting relaxation is called the \textit{master problem} (MP) and it has the following form:
\begin{equation}
	\begin{aligned}
		\min_y \quad & g(z) + \sum_{\omega \in \Omega}\theta_{\omega} \\ 
		\st \quad & \theta \in \mathbb{R}^{\lvert\Omega\rvert},\, z \in Z.
	\end{aligned}
\end{equation}

If the range of $\Theta_\omega$ is $\{0, \infty\}$, then we can omit the corresponding variable. By assumption, MP is significantly easier to solve than OP and yields a lower bound on its optimal objective value. Currently this is obvious since MP is unbounded---nothing links $\theta$ and $z$---but we make sure the property is preserved by Benders cuts, which will be defined soon. For now we can clarify that by the ``underlying optimisation framework'' we mean the solver or algorithm used to solve MP. If OP is an ``almost integer program'' in the above sense, then MP is a MIP which can be solved using the branch-and-bound algorithm.

Let $\theta^*, z^*$ be an optimal solution to MP. We solve the subproblem associated with $z^*$ by calculating $\Theta_{\omega}(z^*)$ for each $\omega \in \Omega$, whether that is done by solving optimisation problems, verifying logical propositions, running simulations, something else, or several things in combination. Then we consider the following case distinctions:
\begin{enumerate}
	\item If $\theta_\omega^* = \Theta_{\omega}(z^*) < \infty$ for all $\omega \in \Omega$, then $z^*$ is an optimal solution to OP.
	\item For each $\omega \in \Omega$ such that $\theta_{\omega}^* < \Theta_{\omega}(z^*) < \infty$, add a Benders optimality cut of the form 
	\begin{equation*}
		B_{\omega z^*}(z) \leq \theta_\omega
	\end{equation*}to MP and continue.
	\item For each $\omega \in \Omega$ such that $\Theta_\omega(z^*) = \infty$, find a function $F_{\omega z^*}$ such that $F_{\omega z^*}(z)\leq 0$ if $z$ is feasible for OP but $F_{\omega z^*}(z^*) > 0$.
	Then add a Benders feasibility cut of the form 
	\begin{equation*}
		F_{\omega z^*}(z) \leq 0
	\end{equation*}
	to MP and continue.
\end{enumerate}
A trivial optimality cut is given by
\begin{equation*}
\theta_\omega \geq B_{\omega z^*}(z) = \begin{cases}
\Theta_\omega(z^*) & \text{if }z = z^*, \\
-\infty & \text{otherwise}.
\end{cases}
\end{equation*}
Cuts like these---and those equivalent to them---are called no-good cuts; they update an estimate at the current solution, but only that solution.
A no-good feasibility cut separates $Z\setminus\{z^*\}$ from $z^*$. In other words, it cuts off the current ``no good'' solution, but only that solution.
In $0$-$1$-programming there is always a linear no-good cut expressing that the value of at least one variable must change.

Benders decomposition iterates between solving the master problem and evaluating solutions.
When the $\theta_\omega$ variables correctly estimate the subproblem values at an optimal master solution, we can stop. Otherwise we add optimality and feasibility cuts as necessary and solve the new master problem. We continue in this way until the subproblem solutions agree with the master estimates.

In the 1980s, researchers such as \cite{PadbergRinaldi1987} started to integrate cutting planes into the branch-and-bound algorithm. Padberg and Rinaldi coined the expression ``branch-and-cut'' for the imposition of new inequalities at nodes of the branch-and-bound tree. Since then, cutting planes have been a feature of all decent MIP solvers.
Branch-and-check is the name often given to a a variant of LBBD that evaluates each incumbent solution to a single master problem; in other words, when Benders cuts are embedded into a branch-and-cut framework; \cite{Erlendur2001}. While branch-and-check does not universally outperform standard LBBD (see \cite{Beck2010}), it often enjoys a number of advantages. Since rebuilding a branch-and-bound tree from scratch in each iteration is often cumbersome and wasteful, in practice, we prefer to add Benders cuts as lazy constraints during a callback routine. Lazy constraints have been available in modern solvers such as Gurobi since 2012; \cite{gurobi}. Standard LBBD may be preferable when the master problem is much easier to solve than the subproblems.

By definition, a Benders cut must be tight at the current solution, and valid at all others. In the worst-case scenario, we can use no-good cuts. ``Logic-based'' Benders decomposition gets its name from the important role of logical inference in deriving stronger cuts. The strength of an optimality cut is determined by how tight the inequality (\ref{eq:benders_qn}) is at other solutions, while the strength of a feasibility cut corresponds to how many other infeasible solutions it eliminates; that is $\lvert \{z \in Z : F_{\omega z^*}(z) < 0\}\rvert$. The task of the practitioner is to infer from the current subproblem solution, valid bounds on as many other solutions as possible.

When MP is a MIP, it is easy to see why stronger cuts translate into a faster branch-and-check algorithm. Consider the branch-and-bound tree of the master problem. A priori, the best known lower bound for each solution is trivial. Therefore, with no-good cuts, we have to build the entire tree in order to prove an optimal solution. But if our cuts impose non trivial bounds at other solutions as well, then we already have useful lower bounds at other nodes. Therefore, we may be able to fathom some nodes without explicitly enumerating all of their child nodes. The more solutions we can infer valid bounds for at a time, and the tighter those bounds are, the faster we can prune the branch-and-bound tree and prove optimality for MP. Similar reasoning also applies to a standard LBBD algorithm.

Having summarised the LBBD method, in the next section we will apply it to the main problem.

\section{Formulation of the Main Problem}
\label{sec:LBBD}In this section we will describe an exact non-linear integer programming formulation of the problem and an LBBD-based solution approach.
The formulation is motivated in part by the following observation;
the only variables in the MIP model (see \ref{sec:MIP}) that represent genuine \textit{decisions} are those that allocate suppression resources.
The rest of the model only evaluates the feasibility and objective value of those decisions.
But if we fix the genuine decisions, their evaluation is almost trivial.

For simplicity we assume there is a single ignition node, say $I_0$. This assumption can be made without loss of generality using a standard super-source transformation. Now, towards the formulation of the problem, for each $n \in N,\, n \neq I_0$ and each $t \in T$, we introduce a binary variable, $z_{nt}$, with $z_{nt} = 1$ if a resource is located at node $n$ at time $t$, and $z_{nt} = 0$ otherwise.
Then define the following constraints:
\begin{align}
\sum_{n \in N}z_{nt} &\leq a_t & & \forall n \in N,\, n \neq I_0, \label{zcon1}\\
\sum_{t \in T}z_{nt} &\leq 1 & & \forall t \in T, \label{zcon2}\\
z_{nt} &\in \{0, 1\} & & \forall n \in N,\,n \neq I_0,\, t \in T. \label{zcon3}
\end{align}
Constraint (\ref{zcon1}) limits the number of available resources, and constraint (\ref{zcon2}) says we can only locate at most one resource at each node over the time horizon.
Let $Z = \{z : (\ref{zcon1}-\ref{zcon3})\}$.
For each $z \in Z$ we define modified weights $c^z$ by
\begin{equation*}
c_{ij}^z = \begin{cases}
c_{ij} + \Delta & \textup{if }\exists t \in T\textup{ with }z_{it} = 1, \\
c_{ij} & \text{otherwise,}
\end{cases}\quad\text{for }ij \in A.
\end{equation*}

\medskip
\noindent In other words, the length of an arc changes if there is a resource on its tail at any time. If $z$ is feasible with respect to Assumption \ref{assume:notonfire}, then $c^z$ describes the correct fire travel times given the interdiction policy $z$.

For each $z \in Z$, let $D(z)$ be a shortest path tree in $(N, A,\, c^z)$; in other words, for each $n \in N$, the unique path from $I_0$ to $n$ in $D(z)$ is the shortest path from $I_0$ to $n$ in $(N, A,\, c^z)$.
The tree $D(z)$ characterises the spread of the fire given $z$.
We let $d(z, n)$ denote the distance from $I$ to $n$ in $D(z)$, that is, with the modified travel times.
For each $n \in N$ and $z \in Z$, we refer to the path from $I_0$ to $n$ in $D(z)$ as the ``binding fire path.'' Consider again the example from the introduction:
\begin{center}
\includegraphics[width=.25\linewidth]{fireSmall1}\qquad\qquad
\includegraphics[width=.25\linewidth]{fireSmall2}
\end{center}
The arcs shown describe the binding fire paths, and the labels are $d(z, n)$ for $n \in N$.

For each $n \in N$, define a function $\Theta_n : Z \to \{0, 1\}$ by 
\begin{equation*}
	\Theta_n(z) = \begin{cases}
		1 & \text{if }d(z, n) < \psi, \\
		0 & \text{otherwise}.
	\end{cases}
\end{equation*}
So $\Theta_n(z) = 0$ if node $n$ is protected by $z$, and $\Theta_n(z) = 1$ otherwise.
Next, for each $n \in N$, define $\phi_n : Z \to \{0, \infty\}$ by
\begin{equation*}
	\varphi_n(z) = \begin{cases}
		\infty & \text{if }\exists t \in T \text{ with }z_{nt}=1\text{ and }d(z, n) < t, \\
		0 & \text{otherwise}.
	\end{cases}
\end{equation*}
In other words, $\phi_n(z) = \infty$ if we allocate a resource to a node that is already on fire, and $\phi_n(z) = 0$ otherwise.
To match Mendes and Alvelos, we are allowed to locate a resource at a node the \textit{exact} instant the fire arrives, otherwise we could define $\phi_n$ with a non-strict inequality.
An exact non-linear integer programming formulation is now
\begin{equation}
\min_z \quad  \sum_{n \in N}\Theta_n(z) + \sum_{n \in N}\phi_n(z) \quad \st \quad z \in Z. \tag{\textrm{OP}}
\end{equation}

\medskip
We omit the $\phi_n$ functions, and for each $n \in N$ we introduce a new continuous variable, $\theta_n$, to linearly approximate the contribution of $\Theta_n(z)$ to the objective function.
The initial form of the master problem is then
\begin{equation}
\begin{aligned}
\min \quad & \sum_{n \in N}\theta_n \\
\st \quad & 0 \leq \theta_n \leq 1 \quad \forall n \in N, \\
& z \in Z.
\end{aligned}\tag{\textrm{MP}}
\end{equation}

\bigskip
Let $\theta^*, z^*$ be an optimal solution to MP. 
We solve the subproblem associated with $z^*$ by calculating $\Theta_n(z^*)$ and $\phi_n(z^*)$ for each $n \in N$. 
This can be done with a single source shortest path algorithm, such as \cite{Dijkstra1959}'s.
If $\theta_n^* = \Theta_n(z^*)$ and $\phi_n(z^*) = 0$ for each $n \in N$, then we are done, and $z^*$ is also an optimal solution to OP.
Otherwise we update MP with Benders cuts and continue.

\subsection{Benders Cuts}

For each $n \in N$, $n \neq I_0$, and $z \in Z$, let $P_n(z) \subset N\setminus(\{n, I_0\})$ denote the nodes on the interior of the binding fire path to $n$.
We can store the paths for the current solution while solving the subproblem.
Let $\theta', z'$ be an incumbent solution to MP such that $\theta'$ is optimal given $z'$.
Suppose $\Theta_n(z') = 1$ for some $n \in N$.
Then $n$ is not protected by $z'$.
But we can infer from this that $n$ will also not be protected by any solution that does not successfully interdict at least one more node in $P_n(z')$.
This translates into an optimality cut.

\begin{lemma}
For each $z' \in Z$ and $n \in N$ with $\Theta_n(z') = 1$,
\begin{equation*}
\Theta_n(z) \geq 1 - \sum_{n' \in P_n(z')}\sum_{
\substack{
t \in T \\ :d(z', I, n') \geq t \\ \wedge z_{n't}' = 0
}}z_{n't} \quad \left(\;=: B_{nz'}(z)\;\right)
\end{equation*}
is a valid inequality on $Z$ and obtains equality at $z'$.
\end{lemma}
\begin{proof}
Let $z \in Z$. Then $B_{nz'}(z)$ is $0$ unless we successfully interdict at least one new node in $P_n(z')$. Otherwise it is non-positive and the inequality is satisfied.
\end{proof}

We can possibly strengthen the Benders cut by noting that it may take more than one new interdiction to intercept the fire in time.
For each $n \in N$ and $z \in Z$ with $\Theta_n(z) = 1$, let
\begin{equation}
\mathcal R_n(z) =
\lceil (\psi - d(z, n))/\Delta \rceil
\end{equation}
be the minimum number of successful interdictions needed to protect $n$.
Since we define $\mathcal R_n(z)$ only for $z$ with $\Theta_n(z) = 1$, we have $\mathcal R_n(z) \geq 1$.
Note that when we interdict the first of $\mathcal R_n(z)$ new nodes, the time window for the other $\mathcal R_n(z) - 1$ interdictions increases by $\Delta$ on the appropriate nodes.
Likewise, the time window for the $\mathcal R_n(z)\ts{th}$ interdiction is increased by $(\mathcal R_n(z^*) - 1)\Delta$ on the appropriate nodes.
But we do not know a priori which nodes the model will interdict at what times.
Thus the new Benders cut has more $z_{nt}$ variables, but smaller coefficients.
\begin{lemma}
For each $z' \in Z$ and $n \in N$ with $\Theta_n(z') = 1$,
\begin{equation}
	\Theta_n(z) \geq 1 - \sum_{n' \in P_n(z')}\sum_{
		\substack{
			t \in T \\ :d(z', I, n') + (\mathcal R_n(z') - 1)\Delta \geq t \\ \wedge z_{n't}' = 0
	}}z_{n't} / \mathcal R_n(z')
\end{equation}
is a valid inequality on $Z$ and obtains equality at $z'$.
\end{lemma}
\begin{proof}
Let $z \in Z$. If (compared to $z^*$) we added fewer than $\mathcal R_n(z')$ new interdictions to $P_n(z')$, then the right hand side is strictly positive. But in that case, $\Theta_n(z) = 1$, so the inequality is satisfied. If we interdict at least $\mathcal R_n(z')$ new nodes on $P_n(z')$, then the right hand side is non-positive and the inequality is satisfied.
\end{proof}

Now suppose $\phi_n(z')$ for some $n \in N$. Then $z'$ is infeasible with respect to Assumption \ref{assume:notonfire}. But we know that in order to get a feasible solution, we must do at least one of two things. We either remove the offending resource, or interdict enough new nodes on the binding fire path early enough to catch the fire. Let $t' \in T$ be the time period with $z_{nt'}' = 1$ and define
\begin{equation}
	\mathcal Q_n(z') = \lceil (t' - d(z', I, n))/\Delta \rceil
\end{equation}
The previous observation then translates into a feasibility cut.
\begin{lemma}
Let $z' \in Z$ and $n \in N$ and suppose $\phi_n(z') = \infty$. Let $t' \in T$ be the time period with $z_{nt'}' = 1$. Then
\begin{equation} \label{cut:feas}
1 - z_{nt'} + \sum_{n' \in P_n(z')}\sum_{
	\substack{
		t \in T \\ :d(z', I, n) + (\mathcal Q_n(z') - 1)\Delta \geq t \\ \wedge z_{n't}' = 0
}}z_{n't}/\mathcal Q_n(z') \geq 1
\end{equation}
is a valid feasibility cut.
\end{lemma}
\begin{proof}
If we do not remove the resource from node $n$ at time $t'$, or add at least $\mathcal Q_n(z')$ new interdictions to $P_n(z')$, then left hand side is strictly less than $1$, which is a contradiction. If we do at least one of those things, then the left hand side is at least $1$ and the inequality is satisfied.
\end{proof}

Taking $z_{nt'}$ to the right-hand side and cancelling $1$ from each side gives another interpretation; that $z_{nt'}$ may not be $1$ unless we interdict at least one new node on the binding fire path.
Now
having derived both feasibility and optimality cuts, we can summarise the formulation.
After adding the optimality cuts associated with $z = 0$ as first-order constraints, we solve MP with Algorithm \ref{callback} as a callback routine.

A natural greedy approach to solving problems with multiple time periods is to use a rolling horizon approach. Here we describe a simple version with no look-ahead.
For each time period, we solve the restricted problem where we assume there are no resources available in future time periods.
We then fix the interdiction decisions for that time period, unfix the other variables, and move on.
For each $t \in T,\,t \neq \max T$, let $\textsf{next}(t) = \min\{t' \in T: t' > t\}$ denote the next time period that resources become available.
The greedy approach is summarised in Algorithm \ref{greedy}. And while the heuristic is not guaranteed to find near-optimal solutions, it  quickly provides us with a sensible starting solution to the master problem.

\begin{algorithm}[H]\caption{Callback Routine for an Incumbent Solution $\theta', z'$ to MP}\label{callback}
\begin{algorithmic}[1]
\medskip
\Require $z'$ integral, $\theta'$ optimal given $z'$
\State Calculate $\Theta_n(z')$ and $\phi_n(z')$ for all $n \in N$
\For{$n \in N$}
	\If{$\varphi_n(z') = \infty$}
		\State $\mathcal Q \longleftarrow \lceil (t' - d(z', I, n))/\Delta \rceil$
		\State Add the following lazy constraint to MP
		\begin{equation*}
		1 - z_{nt'} + \sum_{n' \in P_n(z')}\sum_{
			\substack{
				t \in T \\ :d(z', I, n) + (\mathcal Q - 1)\Delta \geq t \\ \wedge z_{n't}' = 0
		}}z_{n't}/\mathcal Q \geq 1
		\end{equation*}
	\EndIf
	\If{$\theta_n' < 1$ and $\Theta_n(z') = 1$}
		\State $\mathcal R \longleftarrow \lceil (\psi - d(z', I, n))/\Delta \rceil$
		\State Add the following lazy constraint to MP
		\begin{equation*}
			\theta_n \geq 1 - \sum_{n' \in P_n(z')}\sum_{
				\substack{
					t \in T \\ :d(z', I, n') + (\mathcal R - 1)\Delta \geq t \\ \wedge z_{n't}' = 0
			}}z_{n't} / \mathcal R
		\end{equation*}
	\EndIf
\EndFor
\State Continue solving MP
\end{algorithmic}
\end{algorithm}
\begin{algorithm}[H]\caption{Greedy Heuristic}\label{greedy}
\begin{algorithmic}[1]
	\State Put $t = \min T$
	\While{$t \leq \max T$}
	\State Fix $a_{t'} = 0$ for $t' \in T$, $t' > t$
	\State Solve MP with Algorithm \ref{callback} and retrieve optimal $z_{nt}^*$ for $n \in N$
	\State Fix $z_{nt} = z_{nt}^*$ for $n \in N$
	\If{$t < \max T$}
	\State Reset $a$ and put $t \longleftarrow \textsf{next}(t)$
	\EndIf
	\EndWhile
	\State \textbf{Return} $z$
\end{algorithmic}
\end{algorithm}

\section{Computational Experiments}
\label{sec:compExperiments}In this section we describe the results of some computational experiments.
All implementations in this paper were programmed in Python 3.8.12 via the Anaconda distribution (4.11.0).
MIPs were solved using Gurobi 9.1.2.
Dijkstra's algorithm was coded from scratch using the \texttt{heapq} package.
Jobs were run in parallel on a computing cluster operating $2.5$GHz CPUs.
Each job was allocated at most $5$GB of memory, $7200$ seconds of runtime, and a single thread.
Note that we were able to achieve faster solution times using a typical desktop PC. Python code and all instances are available in the accompanying Github repository.

First we compare the exact LBBD approach to the direct MIP model (see \ref{sec:MIP}).
Following Mendes and Alvelos, we consider square grids of sizes $10\times 10$, $20\times 20$, and $30\times 30$.
Each node has an arc leading to each of its orthogonal neighbours.
Ignitions occur at nodes $(5, 5)$, $(10, 10)$, and $(15, 15)$ respectively.
We have an arrival time target of $\psi = 28$ minutes and an induced delay of $\Delta = 50$ minutes.
Three resources become available at $t = 10$ minutes, and another three at $t = 15$ minutes. 
Integer travel times for arcs are sampled from discrete uniform distributions $U(a, b)$ with endpoints (\textit{inclusive}) determined by the direction of that arc.
Parameters for $24$ instances are illustrated in Table~\ref{tbl:benchmark:params}.

\bigskip
\begin{center}{\small
	\setlength{\tabcolsep}{9pt}
	\begin{tabular}{ccccccc}
		& & \multicolumn{4}{c}{Travel Time Distributions} & \\
		\cmidrule{3-6}
		ID & Size & North & South & East & West & Ignition \\
		\cmidrule(l){1-7}
		0 & $10\times 10$  & $U$(7, 9) & $U$(2, 4) & $U$(4, 6) & $U$(6, 8) & (5, 5) \\
		1 & $\prime\prime$ & $U$(7, 9) & $U$(1, 3) & $U$(4, 6) & $U$(6, 8) & $\prime\prime$ \\
		2 &                & $U$(7, 9) & $U$(2, 4) & $U$(3, 5) & $U$(6, 8) &  \\
		3 &                & $U$(7, 9) & $U$(1, 3) & $U$(3, 5) & $U$(6, 8) &  \\
		4 &                & $U$(7, 9) & $U$(2, 4) & $U$(4, 6) & $U$(4, 6) &  \\
		5 &                & $U$(7, 9) & $U$(1, 3) & $U$(4, 6) & $U$(4, 6) &  \\
		6 & $\prime\prime$ & $U$(7, 9) & $U$(2, 4) & $U$(3, 5) & $U$(3, 5) & $\prime\prime$ \\
		7 & $10\times 10$  & $U$(7, 9) & $U$(1, 3) & $U$(3, 5) & $U$(3, 5) & (5, 5) \\
	8--15 & $20\times 20$  & \multicolumn{4}{c}{\textit{Repeat} 0--7}            & (10, 10) \\
   16--23 & $30\times 30$  & \multicolumn{4}{c}{\textit{Repeat} 0--7}            & (15, 15) \\
		\cmidrule(l){1-7}
	\end{tabular}
	\captionof{table}{Parameters for small instances based on \cite{MendesAlvelos2022}.} 
	\label{tbl:benchmark:params}}
\end{center}

The travel times broadly reflect South and South-easterly winds. Mendes and Alvelos observed that the MIP model failed to find feasible solutions for many of these instances. But we can help the model by pre-processing out nodes which are certain to be protected regardless of suppression resources.
In other words, we solve the subproblem associated with $z = 0$.
We then remove all nodes whose objective contribution is still zero.

Solution information for the pre-processed instances can be found in Table \ref{tbl:benchmark:results},
where we record the number of nodes, objective value, and runtime for both the MIP and LBBD, together with the number of Benders cuts generated.

\begin{center}{\small
	\setlength{\tabcolsep}{9pt}
	\begin{tabular}{c c c c c c c c}
		& & \multicolumn{2}{c}{MIP} & \multicolumn{4}{c}{LBBD}\\
		\cmidrule(r){3-4}\cmidrule(r){5-8}
		ID & $\lvert N\rvert$ & Obj & Time (s) & Obj & Time (s) & Opt Cuts & Feas Cuts  \\
		\cmidrule{1-8}
		0 & 50 & 38 & 0.16 & 38 & 0.01 & 24 & 0 \\
		1 & 53 & 40 & 0.28 & 40 & 0.03 & 41 & 1 \\
		2 & 54 & 43 & 0.29 & 43 & 0.07 & 48 & 0 \\
		3 & 57 & 44 & 0.18 & 44 & 0.02 & 43 & 0 \\
		4 & 61 & 44 & 0.25 & 44 & 0.05 & 51 & 0 \\
		5 & 62 & 47 & 0.41 & 47 & 0.01 & 20 & 0 \\
		6 & 68 & 52 & 1.13 & 52 & 0.03 & 47 & 0 \\
		7 & 69 & 54 & 1.46 & 54 & 0.04 & 72 & 0 \\
		8 & 77 & 48 & 1.86 & 48 & 0.03 & 56 & 0 \\
		9 & 88 & 58 & 2.11 & 58 & 0.02 & 65 & 0 \\
		10 & 86 & 54 & 2.71 & 54 & 0.02 & 63 & 0 \\
		11 & 99 & 68 & 4.95 & 68 & 0.13 & 141 & 0 \\
		12 & 88 & 56 & 10.24 & 56 & 0.08 & 101 & 1 \\
		13 & 91 & 57 & 8.51 & 57 & 0.08 & 108 & 1 \\
		14 & 111 & 77 & 21.42 & 77 & 0.3 & 158 & 1 \\
		15 & 141 & 107 & 87.94 & 107 & 0.4 & 338 & 1 \\
		16 & 73 & 38 & 3.87 & 38 & 0.02 & 42 & 0 \\
		17 & 98 & 60 & 3.39 & 60 & 0.02 & 84 & 0 \\
		18 & 82 & 45 & 3.16 & 45 & 0.03 & 78 & 0 \\
		19 & 124 & 81 & 15.95 & 81 & 0.07 & 96 & 0 \\
		20 & 79 & 51 & 4.61 & 51 & 0.06 & 66 & 0 \\
		21 & 123 & 77 & 17.16 & 77 & 0.05 & 138 & 0 \\
		22 & 106 & 75 & 26.63 & 75 & 0.15 & 187 & 0 \\
		23 & 153 & 102 & 31.05 & 102 & 0.2 & 154 & 0 \\
		\cmidrule(l){1-8}
	\end{tabular}
\captionof{table}{Computational Results on Instances from Table \ref{tbl:benchmark:params}. We report on the number of nodes after pre-processing, the objective value and runtime of the MIP and the LBBD method, as well as the number of Benders cuts generated.} 
\label{tbl:benchmark:results}}
\end{center}

\bigskip
We find that LBBD outperforms the strengthened MIP by around an order of magnitude. We repeated the experiment $9$ more times using different seeds to generate the travel times. The mean and standard deviations of the objective values and runtimes over these new instances can be found in Table \ref{tbl:statistics}, and performance profiles can be found in Table \ref{tbl:profiles}

\begin{center}{\small
		\setlength{\tabcolsep}{9pt}
		\begin{tabular}{c c c c c c c}
			& \multicolumn{2}{c}{Objective} & \multicolumn{2}{c}{MIP} & \multicolumn{2}{c}{LBBD}\\
			\cmidrule(r){2-3}\cmidrule(r){4-5}\cmidrule(r){6-7}
			ID & Av Obj & Obj SD & Av Time & Time SD & Av Time & Time SD \\
			\cmidrule{1-7}
			0 & 35.3 & 1.95 & 0.23 & 0.11 & 0.03 & 0.03 \\
			1 & 40.3 & 1.68 & 0.22 & 0.15 & 0.02 & 0.01 \\
			2 & 41.5 & 1.96 & 0.3 & 0.11 & 0.04 & 0.02 \\
			3 & 44.6 & 1.36 & 0.15 & 0.06 & 0.02 & 0.01 \\
			4 & 44.7 & 1.68 & 0.97 & 0.58 & 0.04 & 0.02 \\
			5 & 47.7 & 1.55 & 0.57 & 0.33 & 0.02 & 0.01 \\
			6 & 51.0 & 2.24 & 1.69 & 0.88 & 0.05 & 0.04 \\
			7 & 55.1 & 2.39 & 0.59 & 0.44 & 0.02 & 0.01 \\
			8 & 40.2 & 3.82 & 1.47 & 0.83 & 0.02 & 0.01 \\
			9 & 60.1 & 6.73 & 4.03 & 2.03 & 0.02 & 0.01 \\
			10 & 51.6 & 4.45 & 4.77 & 2.49 & 0.11 & 0.08 \\
			11 & 69.4 & 6.09 & 7.85 & 3.79 & 0.09 & 0.05 \\
			12 & 52.0 & 6.72 & 4.51 & 2.51 & 0.04 & 0.04 \\
			13 & 71.8 & 7.76 & 14.84 & 16.01 & 0.12 & 0.06 \\
			14 & 72.0 & 8.09 & 15.65 & 6.05 & 0.19 & 0.06 \\
			15 & 107.2 & 8.42 & 68.83 & 33.23 & 0.33 & 0.09 \\
			16 & 40.0 & 2.0 & 1.4 & 1.06 & 0.01 & 0.01 \\
			17 & 58.3 & 6.26 & 9.58 & 11.4 & 0.04 & 0.03 \\
			18 & 47.8 & 6.97 & 4.26 & 2.14 & 0.04 & 0.03 \\
			19 & 67.2 & 10.87 & 20.55 & 19.33 & 0.12 & 0.12 \\
			20 & 51.7 & 6.03 & 5.53 & 2.65 & 0.03 & 0.02 \\
			21 & 77.5 & 3.88 & 28.77 & 22.09 & 0.08 & 0.05 \\
			22 & 74.0 & 5.44 & 31.15 & 14.57 & 0.2 & 0.08 \\
			23 & 105.1 & 8.09 & 146.69 & 78.99 & 0.27 & 0.07 \\
			\cmidrule(l){1-7}
		\end{tabular}
		\captionof{table}{Average and standard deviation in objective values and runtimes over $10$ repetitions of the experiment with distinct seeds.} 
		\label{tbl:statistics}}
\end{center}

\begin{center}{\small
		\includegraphics[width=.75\linewidth]{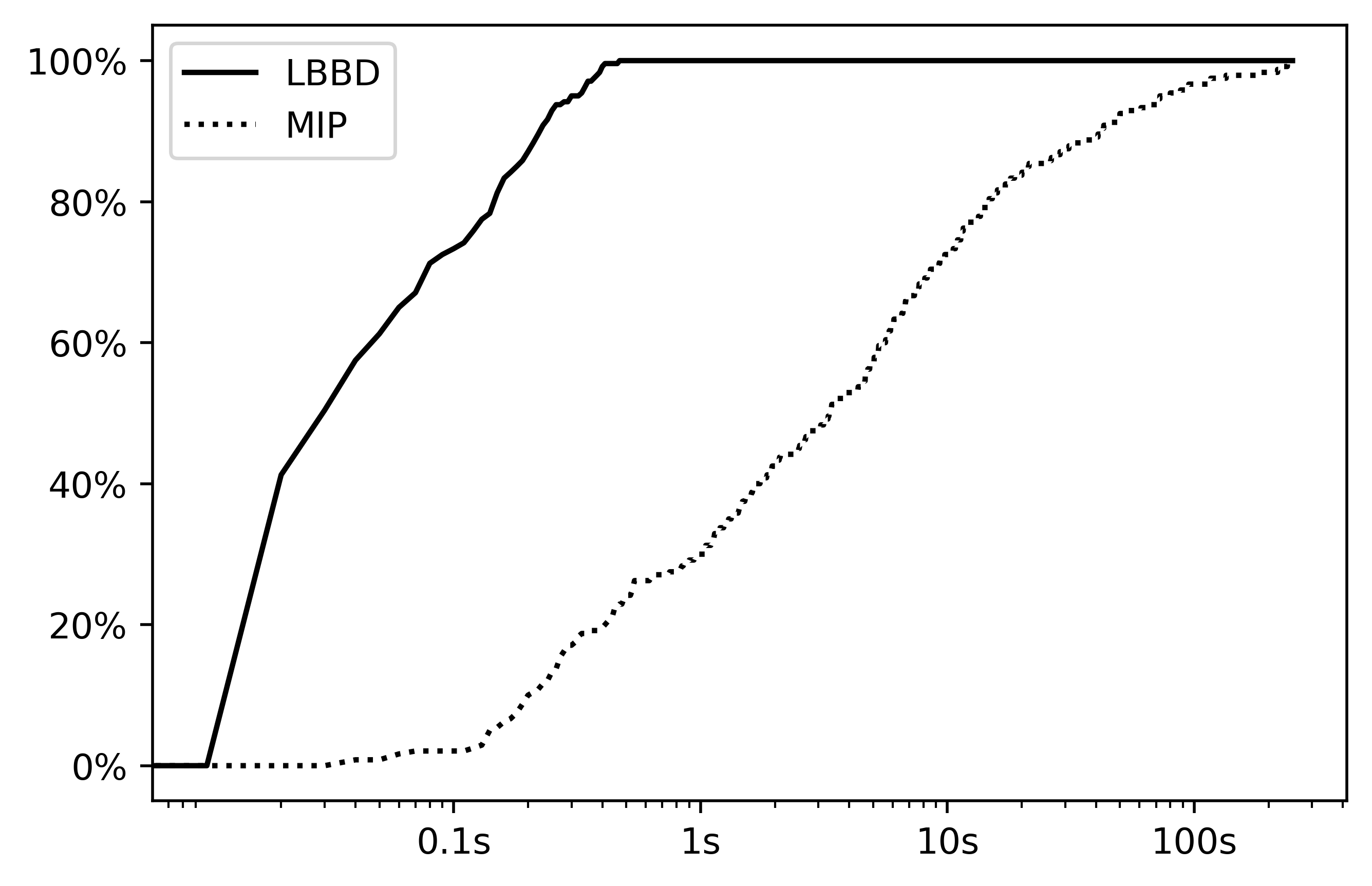}
		\captionof{table}{Logarithmic performance profiles for MIP and LBBD on the $240$ instances described above. The horizontal axis represents time in seconds, and the vertical axis represents the percentage of instances solved to optimality within that time.} 
		\label{tbl:profiles}}
\end{center}

The results of Table \ref{tbl:statistics} make it clear that an arrival time target of $\psi = 28$ is not large enough for the $30\times 30$ instances to be substantially distinct from the $20\times 20$ instances. For the $20\times 20$ instances, already far fewer than the maximum of $400$ nodes remain after pre-processing. See Figure \ref{fig:20by20example}:
\begin{center}{\small
		\includegraphics[width=.55\linewidth]{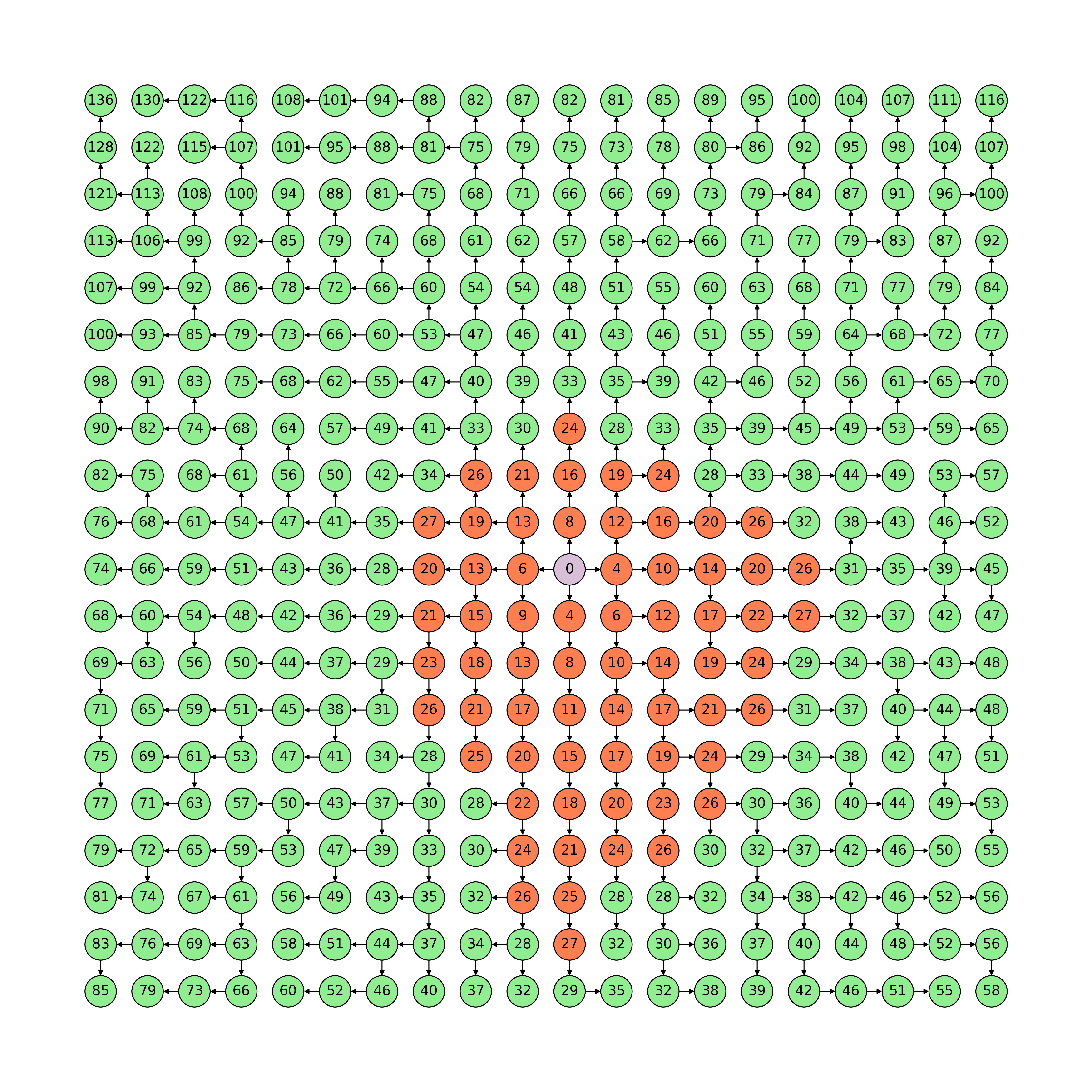}
		\captionof{figure}{The binding fire paths and safe nodes with no interdictions for a typical $20\times 20$ instance.} 
		\label{fig:20by20example}}
\end{center}

In this case, only $16.5\%$ of the nodes remain in the pre-processed network.
To generate more difficult instances, we want to increase the arrival time threshold to the extent that the majority of the network is genuinely at risk. We also want to increase the number of available resources, and therefore the number of feasible solutions. Finally, we can decrease the induced delays, which increases the expected number of new interdictions needed to protect a node. Sixteen new instances---L0A to L7A and L0B to L7B---are described in Table \ref{tbl:benchmark:params2}.

\medskip
\begin{center}{\small
		\setlength{\tabcolsep}{9pt}
		\begin{tabular}{ccccccc}
			& & \multicolumn{4}{c}{Travel Time Distributions} & \\
			\cmidrule{3-6}
			ID & Size & North & South & East & West & Ignition  \\
			\cmidrule(l){1-7}
			L0--7 & $20\times 20$ & \multicolumn{4}{c}{As in Table \ref{tbl:benchmark:params}} & (10, 10) \\
			\cmidrule(l){1-7}
		\end{tabular}
	
		\smallskip
		\begin{tabular}{ccccccccc}
		& \multicolumn{6}{c}{Resources Per Time Period} & & \\
		\cmidrule{2-7}
		Type & $10$ & $20$ & $30$ & $40$ & $50$ & $60$ & $\psi$ & $\Delta$ \\
		\cmidrule(l){1-9}
		A & $3$ & $3$ & $3$ & $3$ & 0 & 0 & $70$ & $50$ \\
		B & $3$ & $3$ & $3$ & $3$ & $3$ & $3$ & $70$ & $30$ \\
		\cmidrule(l){1-9}
		\end{tabular}
		\captionof{table}{Large Instance Parameters.} 
		\label{tbl:benchmark:params2}}
\end{center}

In order to compare it to the LBBD approach, we replicated the ILS metaheuristic of \cite{MendesAlvelos2022}.
One expects the metaheuristic to find decent feasible solutions more rapidly than the exact methods, since MIP solvers spend a significant proportion of the available computational resources tightening the lower bound. But with LBBD it is possible to warm start the master problem with the best known heuristic solution. To ensure a fair comparison, we only warm start the exact LBBD method with the greedy LBBD solution, even if ILS produced a better one. Table \ref{tbl:benchmark:results2} shows the best objective values and runtimes of each method. Of the two heuristics and the two exact methods, the winning objective values have been bolded.

\medskip
\begin{center}{\small
	\setlength{\tabcolsep}{9pt}
	\begin{tabular}{c c c c c c c c c c c c}
		& & \multicolumn{4}{c}{Heuristics} & \multicolumn{6}{c}{Exact methods} \\
		\cmidrule(r){3-6}\cmidrule(r){7-12}
		& & \multicolumn{2}{c}{ILS {(M\&A22)}} & \multicolumn{2}{c}{Gr LBBD} & \multicolumn{3}{c}{Exact LBBD} & \multicolumn{3}{c}{MIP} \\
		\cmidrule(r){3-4}\cmidrule(r){5-6}\cmidrule(r){7-9}\cmidrule(r){10-12}
		ID & $\lvert N\rvert$ & Obj & Time (s) & Obj & Time & Obj & LB & Time & Obj & LB & Time \\
		\cmidrule{1-12} 
		L0A & 289 & \textbf{189} & 255.67 & 193 & 7.56 & \textbf{189} & 189.0 & 156.67 & 236.0 & 111.0 & 7200 \\
		L1A & 294 & 198 & 279.05 & \textbf{189} & 5.29 & \textbf{189} & 189.0 & 44.58 & 245.0 & 143.0 & 7200 \\
		L2A & 282 & \textbf{190} & 228.70 & 193 & 11.94 & \textbf{190} & 190.0 & 144.73 & -- & -- & -- \\
		L3A & 294 & \textbf{222} & 234.73 & 225 & 28.13 & \textbf{207} & 207.0 & 84.47 & 223.0 & 156.0 & 7200 \\
		L4A & 317 & \textbf{216} & 222.56 & 243 & 14.72 & \textbf{216} & 216.0 & 143.84 & -- & -- & -- \\
		L5A & 312 & \textbf{230} & 305.06 & 252 & 34.47 & \textbf{226} & 226.0 & 139.49 & 237.0 & 168.0 & 7200 \\
		L6A & 331 & \textbf{239} & 339.11 & 269 & 59.31 & \textbf{239} & 239.0 & 195.00 & 288.0 & 165.0 & 7200 \\
		L7A & 327 & \textbf{259} & 311.77 & 266 & 28.07 & \textbf{246} & 246.0 & 355.65 & 278.0 & 188.0 & 7200 \\
		L0B & 289 & \textbf{195} & 269.54 & 202 & 13.19 & \textbf{195} & 195.0 & 2515.69 & 215.0 & 103.0 & 7200 \\
		L1B & 294 & 209 & 240.23 & \textbf{201} & 5.47 & \textbf{196} & 196.0 & 488.00  & 206.0 & 129.0 & 7200 \\
		L2B & 282 & 210 & 289.58 & \textbf{199} & 9.89 & \textbf{196} & 196.0 & 3626.87 & 200.0 & 121.0 & 7200 \\
		L3B & 294 & 226 & 442.36 & \textbf{223} & 18.59 & \textbf{213} & 213.0 & 4850.95 & 229.0 & 141.0 & 7200 \\
		L4B & 317 & \textbf{229} & 294.07 & 242 & 11.88 & \textbf{226} & 226.0 & 1790.25 & 262.0 & 126.0 & 7200 \\
		L5B & 312 & 251 & 456.39 & \textbf{245} & 26.19 & \textbf{235} & 235.0 & 5163.17 & -- & -- & -- \\
		L6B & 331 & \textbf{268} & 398.97 & 272 & 28.92 & \textbf{249} & 249.0 & 6264.50 & 279.0 & 168.0 & 7200 \\
		L7B & 327 & \textbf{255} & 378.99 & 268 & 17.50 & \textbf{254} & \textcolor{red}{250.0} & 7200 & 307.0 & 185.0 & 7200 \\
		\cmidrule{1-12}
	\end{tabular}
	\captionof{table}{Computational results for the large instances from Table \ref{tbl:benchmark:params2}. We report the objective value and runtime of (i) our implementation of the ILS metaheuristic of \cite{MendesAlvelos2022}, (ii) the greedy LBBD heuristic, (iii) the exact LBBD algorithm, and (iv) the direct MIP, together with all proven lower bounds. The red text indicates the instance for which optimality was not proved within the time limit.} 
	\label{tbl:benchmark:results2}}
\end{center}

Several observations can now be made. First we note that Mendes and Alvelos' ILS metaheuristic produces better objective values than the greedy LBBD algorithm most of the time. The metaheuristic was even able to find the optimal solution to five instances; something accomplished by the greedy algorithm only once. On the other hand, the exact LBBD algorithm not only found the optimal solution, but proved optimality in all but one instance, L7B. Even then we were able to prove near-optimality using the lower bound. The MIP model failed to achieve a respectable optimality gap for any instance, and exceeded the memory limit three times.

\section{Conclusion}
\label{sec:conclusion}In this paper we have proposed a novel formulation and solution approach for a wildfire suppression problem. The new formulation overcomes a number of the inherent difficulties with the problem, and expands the range of problem sizes which can conceivably be solved to provable optimality. We expect that on truly huge instances, the best approach will be to warm-start the LBBD algorithm with the ILS metaheuristic of Mendes and Alvelos, or a more intricate rolling horizon heuristic with look-ahead. Future work includes applying the new method to instances derived from real landscapes, with parameters estimated based on the underlying fuel load, elevation pattern, and forecast weather. We may also consider more dense networks, perhaps arising from a hexagonal decomposition, or diagonal arcs. It would also be interesting to consider a stochastic extension of the problem that allows the fire travel times to be random variables that depend on the wind, and change over time. There may also be scope for applying LBBD to different but related wildfire suppression problems.

\section*{Acknowledgments}
Mitchell Harris is supported by an Australian Government Research Training Program (RTP) scholarship.

% Bibliography
\addcontentsline{toc}{section}{References}
\bibliography{fireLBBD}

% Appendices
\appendix
\section{A Strengthened MIP Formulation}
\label{sec:MIP}In this section we modify the MIP from \cite{Alvelos2018, MendesAlvelos2022} and eliminate some symmetries.
For simplicity, we assume that there is a single ignition node, say $I_0$. This assumption can be made without loss of generality using the standard super-source trick.
The first difference between our MIP and Alvelos' is that we have removed the $r$ index from the resource allocation variables (the $z$ variables).
This is possible because the suppression resources are identical.
We do not care \textit{which} resource is allocated to a node, only \textit{whether} one is allocated.
An $r$ index can be re-introduced in the case of heterogeneous resources.
We have also omitted a number of inventory variables and constraints that let unused resources be carried over into future time periods.
If the fire has not reached a node yet in time period $t$, then it has not reached that node yet at any earlier time.
Thus, if we would allocate an unused resource to a node in a later time period, there is no reason not to allocate it to that node as soon as the resource becomes available.
We omit a constraint that prevents the $y_n$ variables from erroneously equalling $1$, since that is ensured by virtue of minimisation.
Finally we have omitted an $\eps$ expression from the objective function that would prohibit the allocation of redundant resources. We did this for the sake of an integer objective value, but it is easy to reintroduce.
With these observations in mind, we introduce the following decision variables:
\begin{equation}
\begin{tabular}{rl}
$x_a \in \mathbb{R}_+$ & The number of binding fire paths that use arc $a \in A$, \\
$\lambda_n \in \mathbb{R}_+$ & The arrival time of fire at node $n \in N$, \\
$s_a \in \mathbb{R}_+$ & An explicit slack variable for each $a \in A$, \\
$q_a \in \{0, 1\}$ & $=1$ iff arc $a$ belongs to a binding fire path, \\
$z_{nt} \in \{0, 1\}$ & $=1$ iff there is a resource at node $n \in N,\,n \neq I_0$ at time $t \in T$, \\
$y_n \in \{0, 1\}$ & $=1$ iff node $n \in N$ is not protected.
\end{tabular} \label{mip:variables}
\end{equation}
Put
\begin{equation*}
M = (\lvert N\rvert - 1)\max\{c_a : a \in A\} + \left(\sum_{t \in T}a_t\; - 1\right)\Delta.
\end{equation*}
For each node $n \in N$, $\inarcs(n)$ and $\outarcs(n)$ denote the sets of arcs entering and leaving node $n$ respectively.
A MIP formulation of the problem is given by
{\begin{alignat}{2}
\min \quad & \sum_{n \in N}y_n &\quad & \label{mip:obj}\\
\st \quad &
\sum_{a \in \outarcs(I_0)}x_a = \lvert N\rvert - 1, & & \label{mip:con1}\\
& \sum_{a \in \inarcs(n)}x_a - \sum_{a \in \outarcs(n)}x_a = 1 & & \forall n \in N,\, n \neq I_0, \label{mip:con2}\\
& \lambda_j - \lambda_i + s_{ij} = c_{ij} + \Delta \sum_{t \in T}z_{it} & & \forall ij \in A, \label{mip:con3}\\
& \lambda_{I_0} = 0, \label{mip:con4}\\
& x_a \leq (\lvert N\rvert - 1)q_a & & \forall a \in A, \label{mip:con5}\\
& s_a \leq M(1 - q_a) & & \forall a \in A, \label{mip:con6}\\
& \sum_{n \in N}z_{nt} \leq a_t & & \forall n \in N,\, n \neq I_0, \label{mip:con7}\\
& \sum_{t \in T}z_{nt} \leq 1 & & \forall t \in T, \label{mip:con8}\\
& z_{nt} \leq 1 + \frac{\lambda_n - t}{t} & & \forall n \in N,\,n \neq I_0,\, t \in T, \label{mip:con9} \\
& y_n \geq \frac{\psi - \lambda_n}{\psi} & & \forall n \in N, \label{mip:con10} \\
& \text{Variable domains given by }(\textup{\ref{mip:variables}}). \label{mip:con11}
\end{alignat}}

The objective function minimises the number of nodes which we fail to protect. Constraints (\ref{mip:con1}~and~\ref{mip:con2}) are the constraints of the shortest path tree problem. Constraints (\ref{mip:con3}~and~\ref{mip:con4}) are the dual constraints associated with the shortest path tree variables with the updated travel times. Constraints (\ref{mip:con5} and \ref{mip:con6}) then enforce complementary slackness and ensure that the $\lambda$ variables represent the correct arrival times. Constraint (\ref{mip:con7}) limits the number of resources available in each time period. Constraint (\ref{mip:con8}) ensures we allocate at most one resource to each node. Constraint (\ref{mip:con9}) ensures that we do not locate a resource at a node that is already on fire. Constraint (\ref{mip:con10}) forces $y_n$ to be $1$ if the fire arrives at $n$ strictly sooner than $\psi$. Constraint (\ref{mip:con11}) states that the domains of the variables are described by (\ref{mip:variables}). We refer the reader to the aforementioned studies for a fuller explanation of this formulation.

\end{document}